\documentclass{lms}
\usepackage{graphicx}
\usepackage{amsmath}
\usepackage{amssymb, amscd}
\usepackage{epstopdf}
\usepackage[all]{xy}
\usepackage{framed}
\DeclareMathOperator{\sech}{sech}
\newtheorem{theorem}{Theorem}[section] 
\newtheorem{lemma}[theorem]{Lemma}     
\newtheorem{corollary}[theorem]{Corollary}
\newtheorem{proposition}[theorem]{Proposition}

\newnumbered{assertion}{Assertion}    
\newnumbered{conjecture}{Conjecture}  
\newnumbered{definition}{Definition}
\newnumbered{hypothesis}{Hypothesis}
\newnumbered{remark}{Remark}
\newnumbered{note}{Note}
\newnumbered{observation}{Observation}
\newnumbered{problem}{Problem}
\newnumbered{question}{Question}
\newnumbered{algorithm}{Algorithm}
\newnumbered{example}{Example}
\newunnumbered{notation}{Notation} 

\def\eproof{$\Box$ \medskip}

\def\chb{{\rm Dome}(\Omega)}

\newcommand\rs{\hat{\mathbb{C}}}
\newcommand\C{\mathbb{C}}

\newcommand\D{\mathbb{D}}
\newcommand\R{\mathbb{R}}
\newcommand\Sp{\mathbb{S}}
\newcommand\U{\mathbb{U}}
\newcommand\Ht{\mathbb{H}^3}
\newcommand\Hp{\mathbb{H}^2}

\newcommand\Z{\mathbb{Z}}
\newcommand\EMM{Epstein-Marden-Markovic}

\title[Improved bound for Sullivan's convex hull theorem]
 {An improved bound for Sullivan's  convex hull theorem} 

\author{M. Bridgeman, R. Canary and A. Yarmola}

\classno{51M10,  (primary),  51N25 (secondary)}

\extraline{Bridgeman was partially supported by  grant \#266344 from the Simons Foundation,
Canary was partially supported by NSF grant DMS -1306992.}

\begin{document}
\maketitle

\begin{abstract}
Sullivan showed that there exists $K_0$ such that if $\Omega\subset \rs$ is a simply connected hyperbolic 
domain, then there exists a conformally natural
$K_0$-quasiconformal map from $\Omega$ to the boundary $\chb$ of
the convex hull of its complement which extends to the identity on $\partial\Omega$. Explicit upper and lower
bounds on $K_0$ were obtained by Epstein, Marden, Markovic and Bishop. We improve on these bounds,
by showing that one may choose $K_0\le 7.1695$.
\end{abstract}

\section{Introduction}

In this paper we consider the relationship between the Poincar\'e metric on a hyperbolic simply connected domain 
$\Omega$ in $\rs = \partial \Ht$ and the geometry of the boundary $\chb$ of the convex core of its complement in $\Ht$. 
Sullivan \cite{sullivan} (see also Epstein-Marden \cite{EM87}) 
showed that  there exists $K_0  > 0$ such if $\Omega$ is simply connected,
then there is a conformally natural \hbox{$K_0$-quasiconformal} map $f:\Omega \rightarrow \chb$  
which extends to the identity on $\partial\Omega$.  Epstein, Marden and Markovic provided upper and lower bounds
for the value of $K_0$.

\begin{theorem}{\rm (Epstein-Marden-Markovic \cite{EMM1,EMM2})}
There exists $K_0 \leq 13.88$ such that if $\Omega\subset\rs$ is a simply connected hyperbolic domain, then there is
a conformally natural \hbox{$K_0$-quasiconformal} map  $f:\Omega\to\chb$ which extends continuously to the identity on 
$\partial \Omega\subset\rs$.
Moreover, one may not choose $K_0\le 2.1$.
\label{emmKeq}
\end{theorem}

We recall that $f$ is said to be {\em conformally natural} if whenever $A$ is a conformal automorphism of $\rs$
which preserves $\Omega$,
then $\bar A\circ f=f\circ\Omega$ where $\bar A$ is the extension of $A$ to an isometry of $\Ht$.
If one does not require that the quasiconformal map $f:\Omega\to\chb$ be conformally natural,
Bishop \cite{bishop} obtained a better uniform bound on the quasiconformality constant. Epstein and Markovic \cite{EM05}
showed that even in this setting one cannot uniformly bound the quasiconformality constant above by 2.

\begin{theorem}{\rm (Bishop \cite{bishop})}
There exists $K_1 \leq 7.88$ such that if $\Omega\subset\rs$ is a simply connected hyperbolic domain,
then there is a $K_1$-quasiconformal map $f:\Omega\to\chb$ which extends continuously to the identity on 
$\partial \Omega\subset\rs$.

\end{theorem}

In this paper, we obtain a bound in the conformally natural setting, which improves on both of these bounds.

\begin{theorem}
\label{main}
There exists  $K_0 \leq 7.1695$ such that if $\Omega\subset\rs$ is a simply connected hyperbolic domain,
then there is a conformally natural \hbox{$K_0$-quasiconformal} map \hbox{$f:\Omega\to\chb$}
which extends continuously to the identity on $\partial \Omega\subset\rs$.
\end{theorem}

\medskip\noindent
{\em Outline of argument:}
One may realize $\chb$ as the image of a pleated plane $P_\mu:\Hp\to\Ht$ whose bending
is encoded by a measured lamination $\mu$. Given $L>0$, we define the $L$-roundness  $||\mu||_L$ of $\mu$
to be the least upper bound on the total bending of $P_\mu(\alpha)$ where $\alpha$ is an open geodesic segment in $\Hp$ of
length $L$.  (This generalizes the notion of roundness introduced by Epstein-Marden-Markovic \cite{EMM1}.)
Our first bound improves on an earlier bounds of Bridgeman \cite{bridgeman,bridgeman2} on roundness.

\medskip\noindent
{\bf Theorem \ref{bendboundbetter}.} {\em 
If $L \in (0, 2 \sinh^{-1}(1))$, $\mu$ is a measured lamination on $\Hp$ and $P_\mu$ is an embedding,
then
$$||\mu||_L\le 2 \cos^{-1}\left(-\sinh\left(\frac{L}{2}\right)\right).$$
}

\medskip

We then generalize  work of Epstein-Marden-Markovic \cite[Theorem 4.2, part 2]{EMM1} and an unpublished result of
Epstein and Jerrard  \cite{EJpre} which give criteria  for
$P_\mu$ to be an embedding.

\medskip\noindent
{\bf Theorem \ref{ej}.} {\em There exists an increasing function  $G:(0,\infty) \rightarrow (0,\pi)$ 
with $G(1) \approx 0.948$, such that if
$\mu$ is a measured lamination on $\Hp$ such that 
$$||\mu||_L < G(L),$$ 
then $P_\mu$ is a bilipschitz embedding
which extends continuously to  a map
\hbox{$\hat P_\mu:\Hp\cup \Sp^1\to \Ht\cup\rs$} so that  $\hat P_\mu(\Sp^1)$ is a quasi-circle.
}

\medskip

With these bounds in place, we may adapt the techniques of Epstein, Marden and Markovic \cite{EMM1,EMM2}
to complete the proof of our main result.

\section{Pleated planes and $L$-roundness}

In this section, we recall the definition of the pleated plane associated to a measured lamination, and introduce
the notion of $L$-roundness.

Let $G(\Hp)$ be the set of unoriented geodesics on the hyperbolic plane $\Hp$.
One may identify $G(\Hp)$ with $(\Sp^1\times\Sp^1-\Delta)/\Z_2$.
A {\em geodesic lamination} on $\Hp$ is a closed subset $\lambda \subset G(\Hp)$
which does not contain any intersecting geodesics. 
A {\em measured lamination} $\mu$ on $\Hp$ is a non-negative measure $\mu$ on $G(\Hp)$ supported on a geodesic lamination 
$\lambda = \text{supp}(\mu)$. A geodesic arc $\alpha$ in $\Hp$ is said to be transverse 
to $\mu$, if it is transverse to every geodesic
in the support of $\mu$.
If $\alpha$ is transverse to $\mu$, we define
$$i(\mu,\alpha) = \mu \left( \{ \gamma \in G(\Hp) \mid \gamma \cap \alpha\ne\emptyset\} \right).$$ 
If $\alpha$ is not transverse to $\mu$, then it is contained in a geodesic in $\text{supp}(\mu)$ and we let $i(\mu,\alpha)=0$.

Given a measured lamination $\mu$ on $\Hp$, we may define a pleated plane \hbox{$P_\mu:\Hp\to\Ht$},
well-defined up to post-composition by an
isometry of $\Ht$. $P_\mu$ is an isometry on the components
of $\Hp-\text{supp}(\mu)$, which are called flats. 
 If $\mu$ is a finite-leaved lamination, then $P_\mu$ is simply obtained by bending, consistently rightward, by the angle $\mu(l)$
along each leaf $l$ of $\mu$. Since any measured lamination is a limit of finite-leaved laminations, one may define $P_\mu$ in
general by taking limits (see \cite[Theorem 3.11.9]{EM87}).

If $\Omega\subset\rs$ is a simply connected hyperbolic domain, let $\chb$ denote the 
boundary of the convex hull of its complement $\rs-\Omega$. 
Thurston \cite{ThBook}  showed that there exists a lamination $\mu$ on $\Hp$ such that
$\chb=P_\mu(\Hp)$ and  $P_\mu:\Hp\to\chb$ is an isometry. (See
Epstein-Marden \cite[Chapter 1]{EM87},
 especially sections 1.11 and 1.12,  for a detailed exposition.)

\begin{lemma}
\label{pleated plane and dome}
If $\Omega$ is a hyperbolic domain, there is a lamination $\mu$ on $\Hp$ such that
$P_{\mu}$ is a locally isometric covering
map with image $\chb$.
\end{lemma}

For any point $p \in \chb$, a {\em support plane} at $p$ is a totally geodesic plane through $p$ which is disjoint from the interior
of the convex hull of $\rs-\Omega$.
The exterior angle, denoted  $\angle(P,Q)$,
between two intersecting support planes $P$ and $Q$ is the angle between their normal vectors at a point of intersection. 

Let  $\alpha:[a,b]\to\Hp$ be a unit-speed closed geodesic arc.
If $\alpha(t)$ lies on a leaf $l$ of $\mu$ with
$\mu(l)>0$, then there is a maximal family  $\{Q_l^\theta\}_{\theta\in [0,\mu(l)]}$ of support planes to $\chb$
through $P_\mu(\alpha(t))$, all of which contain $P_\mu(l)$. In all other cases, $\chb$ has  a unique support plane at $P_\mu(\alpha(t))$.
One may concatenate all the support planes to points in $P_\mu(\alpha([a,b]))$ to obtain 
a continuous family $\{ P_t\}_{t\in [0,k]}$ of support planes along $\alpha$, so that
$P_0$ is the leftmost support plane to $\chb$ at $P_\mu(\alpha(a))$ and
$P_k$ is the rightmost support plane to $\chb$ at $P_\mu(\alpha(b))$.
Moreover, there exists a continuous non-decreasing
function $q:[0,k]\to [a,b]$ so that  $P_t$ is a support plane to $\chb$ at $P_\mu(\alpha(q(t)))$ for all $t$.
If $0=t_0<t_1<\cdots<t_n=k$
and $P_t$ intersects both $P_{t_{i-1}}$ and $P_{t_i}$ for all $t\in[t_{i-1},t_i]$, then
$$i(\mu,\alpha) \leq \sum_{i = 1}^n \angle(P_{t_{i-1}}, P_{t_i}).$$
See Section 4 of \cite{BC03}, especially Lemma 4.1, for a more careful discussion.

For a measured lamination $\mu$ on $\Hp$, Epstein, Marden and Markovic \cite{EMM1} defined the 
{\em roundness}  of $\mu$ to be
$$||\mu|| = \sup i(\mu,\alpha)$$
where the supremum is taken over all  open unit length geodesic arcs  in $\Hp$. The roundness bounds
the total  bending of $P_\mu$ on any segment of length 1 and 
is closely related to average bending, which was  introduced earlier by the first author in \cite{bridgeman}.
In this paper, it will be useful to consider the {\em $L$-roundness} of a measured lamination for any $L>0$
$$||\mu||_L = \sup i(\alpha,\mu)$$
where  now the supremum is taken over all open geodesic arcs of length $L$ in $\Hp$.  
We note that the supremum over open geodesic arcs of length $L$, 
is the same as that over half open geodesic arcs of length $L$.

In \cite{bridgeman2}, the first author obtained an upper bound on the  $L$-roundness of  an embedded pleated plane. 

\begin{theorem}{\rm (Bridgeman \cite{bridgeman2})}
There exists a strictly increasing homeomorphism \hbox{$F:[0,2\sinh^{-1}(1)]\rightarrow [\pi,2\pi]$}
such that if $\mu$ is a measured lamination on $\Hp$ and $P_\mu$ is an embedding,  then  
$$||\mu||_L \le F(L)$$
 for all $L\le 2\sinh^{-1}(1)$. In particular, 
$$||\mu||\le F(1)=2\pi-2\sin^{-1}\left(\frac{1}{\cosh(1)}\right)\approx 4.8731.$$
\label{bendbound}
\end{theorem}

Epstein, Marden and Markovic \cite{EMM1} provided a criterion guaranteeing that a pleated plane
is a bilipschitz embedding.

\begin{theorem}{\rm (\EMM\ \cite[Theorem 4.2, part 2]{EMM1})}
If $\mu$ is a measured lamination on $\Hp$ such that $||\mu|| \leq c_2=73$,
then   $P_\mu$ is a bilipschitz embedding which extends to an embedding $\hat P_\mu:\Hp\cup \Sp^1\to \Ht\cup\rs$ 
such that  $\hat P_\mu(\Sp^1)$ is a quasi-circle. 
\label{sufficient}
\end{theorem}

In  \cite{EMM2}, Epstein, Marden and Markovic comment 
``Unpublished work by David Epstein and Dick Jerrard should prove that $c_2 > .948$, 
though detailed proofs have not yet been written''. 
The authors contacted David Epstein who kindly provided their notes outlining the proof.
In section \ref{embedding} we prove a generalization of their result 
using the approach outlined in their notes.

\section{An upper bound on $L$-roundness for embedded pleated planes}

In this section, we adapt the
techniques of  \cite{bridgeman2} to obtain an improved bound on  the \hbox{$L$-roundness}
of an embedded pleated plane.

\begin{theorem}\label{bendboundbetter}
If $L \in (0, 2 \sinh^{-1}(1))$, $\mu$ is a measured lamination on $\Hp$ and $P_\mu$ is an embedding,
then
$$||\mu||_L\le c_1(L)=2 \cos^{-1}\left(-\sinh\left(\frac{L}{2}\right)\right).$$
\end{theorem}

\begin{proof}
Since $F(2\sinh^{-1}(1)) = 2\pi $, Theorem \ref{bendboundbetter} follows from Theorem \ref{bendbound}
when \hbox{$L=2\sinh^{-1}(1).$} Therefore, we may assume that \hbox{$L < 2\sinh^{-1}(1)$}.

Let $\alpha:[0,L] \to \Hp$ be a geodesic arc of length $L < 2\sinh^{-1}(1)$.
Let $\{P_t\ |\ t\in [0,k]\}$ be the continuous one-parameter family of support planes to
$\alpha$ and let  $q:[0,k]\to [0,L]$ be the continuous non-decreasing map such that
$P_t$ is a support plane to $\chb$ at $\alpha(q(t))$ for all $t$.

We now recall  the proof of Lemma 4.3 in \cite{BC03}.
If $P_0$ intersects $P_t$ for all $t\in [0,k)$, then
$i(\alpha,\mu)\le \pi$ and we are done. If not, there exists $a \in (0,k)$ such that $P_a$ 
has an ideal intersection point with $P_0$ and $P_t$ intersects $P_0$ for all $t \in (0,a)$.
If there exists $t\in (a,k]$ so that $P_t$ is disjoint from $P_a$, then Lemma 3.2 in 
\cite{BC03} implies that $\alpha([0,q(t)])$ has length at least $2\sinh^{-1}(1)$, which would
be a contradiction. Therefore, if $t\in (a,k]$, then $P_a$ intersects $P_t$.
One of the key arguments in the proof of \cite[Lemma 4.3]{BC03} gives that $P_0$ must
be disjoint from $P_k$ (since otherwise one could extend $\alpha([0,1])$ to a closed curve by
appending arcs in $P_0\cup P_k$ and then project onto $\chb$ to find a homotopically
non-trivial curve  on $\chb$.)

Let $\phi$ be the interior angle of intersection between $P_a$ and $P_k$.
The interior angle of intersection between $P_t$ and $P_0$ varies continuously from $\pi$ to $0$ as $t$
varies between $0$ and $a$ and achieves the value $0$ only at $a$. 
There exists  $c \in (0,a)$ such that $P_c$ has an ideal intersection with
$P_k$ and $P_t$ intersects $P_k$ for all $t \in (c,a)$ (since otherwise we could again argue
that $i(\mu,\alpha)\le \pi$).
The interior angle of  intersection of $P_t$ with $P_k$ varies from $0$ to $\phi$ as $t$ varies from $c$ to $a$.
Thus, there exists some $b \in (c,a)$ such that $P_b$ intersects $P_0$ and $P_k$ in 
the same interior angle $\theta >0$.
Therefore, by \cite[Lemma 4.1]{BC03}, we have
$$i(\mu,\alpha) \leq  2\pi - 2\theta.$$

Consider the plane $R$ perpendicular to $P_0$, $P_b$ and $P_k$.
Consider the three geodesics $g_s=P_s\cap R$,
where $s = 0$,  $b$ or $k$. Notice that $g_b$ intersects both $g_0$ and $ g_k$ with interior angle $\theta$.
Let $\bar\alpha$ be the orthogonal projection of $\alpha$ to $R$. 
Then $\bar\alpha$ is a curve in $R$ with $\bar\alpha(q(s)) \in g_s$ for $s= 0, b,k$. 
Let $\beta$ be the shortest curve joining a point of $g_0$ to a point on $g_k$ which intersects $g_k$. One
may easily check that $\beta$ consists of two geodesic arcs $\beta_0$ and $\beta_1$ such that $\beta_0$ intersects $g_0$
perpendicularly, $\beta_1$ intersects $g_k$ perpendicularly and $\beta_0$ and $\beta$ make the same
angle with $g_b$ at their common point of intersection.

\begin{figure}[htbp] 
   \centering
   \includegraphics[width=5in]{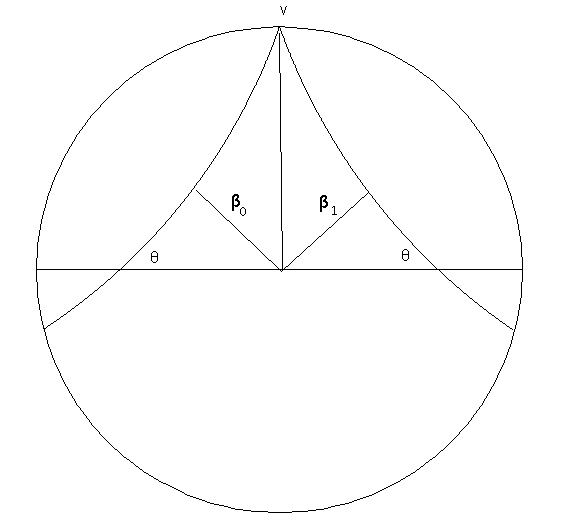} 
   \caption{The triangle $T$ and its decomposition}
   \label{isos1}
\end{figure}

Since $g_0$ and $g_k$ do not intersect, $\beta$ is shortest when the geodesics $g_0$ and $g_k$ 
have a common ideal point. In this case, the geodesics $g_0$, $g_b$ and $ g_k$ form an isosceles triangle $T$ 
with an ideal vertex (see Figure \ref{isos1}).
One may apply hyperbolic trigonometry formulae \cite[Theorem 7.9.1]{beardon} and \cite[Theorem 7.11.2]{beardon}
to check that in this case
$$ \cos(\theta) = \sinh(\ell(\beta)/2).$$
So, in general
$$\ell(\beta) \ge 2\sinh^{-1}(\cos(\theta)).$$
Since, by construction, $\ell(\beta)\le \ell(\alpha)=L$, we see that
$$L \geq 2\sinh^{-1}(\cos(\theta))$$ 
which implies that 
$$\theta \geq \cos^{-1}(\sinh(L/2)).$$
Therefore,
$$i(\mu,\alpha) \leq 2\pi-2\cos^{-1}(\sinh(L/2)) = 2\cos^{-1}(-\sinh(L/2))$$
for any closed geodesic arc $\alpha$ of length $L$. 
Therefore, the same bound holds for all open geodesic arcs of length $L$ and the result follows.
\end{proof}

\section{A new criterion for embeddednes of pleated planes}
\label{embedding}

In this section, we provide a new criterion which guarantees the embeddedness of a pleated plane
which generalizes earlier work of Epstein-Marden-Markovic \cite{EMM1}
(see Theorem \ref{sufficient}) and an unpublished result of
Epstein-Jerrard \cite{EJpre}

\begin{theorem}
There exists an increasing function  $G:(0,\infty) \rightarrow (0,\pi)$, such that if
$\mu$ is a measured lamination on $\Hp$ and
$$||\mu||_L < G(L),$$ 
then $P_\mu$ is a bilipschitz embedding which extends continuously to  a  map
\hbox{$\hat P_\mu:\Hp\cup \Sp^1\to \Ht\cup\rs$}
such that  $\hat P_\mu(\Sp^1)$ is a quasi-circle.
\label{ej}
\end{theorem}

Since $G(1) \approx 0.948$, we recover the result claimed by Epstein and Jerrard as a special case.

\begin{corollary}{\rm (Epstein-Jerrard \cite{EJpre})}
If $\mu$ is a measured lamination on $\Hp$ such that 
$$||\mu|| < .948$$ 
then $P_\mu$ is a bilipschitz embedding 
which extends continuously to  a map
\hbox{$\hat P_\mu:\Hp\cup \Sp^1\to \Ht\cup\rs$} such that the image of $\Sp^1$ is a quasi-circle.
\end{corollary}

We begin by finding an embedding criterion for piecewise geodesics. This portion of the proof follows Epstein and Jerrard's
outline quite closely. Such a criterion is easily translated into a criterion for the embeddedness of pleated planes
associated to finite-leaved laminations. We then further show that, in the finite-leaved lamination case, the pleated
planes are in fact quasi-isometric embeddings with uniform bounds on the quasi-isometry constants. The general
case is handled by approximating a general pleated plane by pleated planes associated to finite-leaved laminations.

\medskip\noindent
\begin{remark}
As in \cite[Theorem 4.2]{EMM1} we can  consider a horocycle $C$ in $\Hp$ and a sequence of points on $C$ with  
hyperbolic distance between consecutive points being $L$. Connecting consecutive points, one obtains an
embedded piecewise geodesic $\gamma$ in  $\Ht$. Let $P_\mu(\Hp)$ be the pleated plane in $\Ht$ obtained by
extending each flat in $\gamma$ to a flat in $\Ht$. One may check that 
$$||\mu||_L= 2\sin^{-1}\left(\tanh\left(\frac{L}{2}\right)\right)$$
which is the conjectured optimal bound.
Since $2\sin^{-1}(\tanh(1/2))\approx .96076$, Theorem \ref{ej} is nearly optimal when $L=1$.
Comparing the bounds for all $L \in [0, 2\sinh^{-1}(1)]$, we see they are also close to optimal (see Figure \ref{optimal}).
\end{remark}

\begin{figure}[hpbt] 
\centering
   \includegraphics[width=5in]{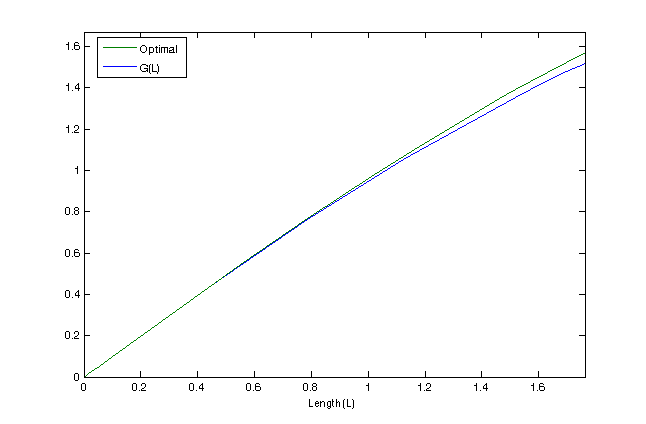} 
   \caption{$G(L)$ and the conjectured optimal bound  $2\sin^{-1}(\tanh(L/2))$ on $[0, 2\sinh^{-1}(1)]$}
   \label{optimal}
\end{figure}

\subsection{Piecewise geodesics}

Let $J$ be an interval in $\mathbb R$ containing $0$.  A continuous map $\gamma:J\to \Ht$ will be called a
``piecewise geodesic'' if there exists a discrete subset $\{t_i\}$ in $J$, parameterized
by an interval in $\Z$, 
such that, for all $i$, $t_i<t_{i+1}$ and $\gamma((t_i,t_{i+1}))$  is a geodesic arc. 
(If  there is a first bending point $t_r$, we let $t_{r-1}=\inf J$ and if there is a last bending point $t_s$, 
we define $t_{s+1}=\sup J$.)
We will
call $t_i$ (or $\gamma(t_i)$) the {\em bending points} of $\gamma$.
The {\em bending angle} $\phi_i$  at  $t_i$ is the angle between $\gamma([t_{i-1},t_i])$ and
$\gamma([t_i,t_{i+1}))$.   Let
$$s(t)=d_{\Ht}(\gamma(0),\gamma(t)).$$

If $L>0$, by analogy with the definition of $L$-roundness, we may define
$||\gamma||_L$  to be the supremum of the total bending angle in any open subsegment of
$\gamma$ of length $L$.

If $t\ne t_i$ for any $i$,  then let 
$\theta(t)$ be the angle between the ray from $\gamma(0)$ to $\gamma(t)$ and the tangent vector
$\gamma'(t)$.
For $i=1,\ldots,n$, we define
$$\gamma_+'(t_i)=\lim_{t\to t_i^+} \gamma'(t)\ \ \ {\rm and}\ \ \
\gamma_-'(t_i)=\lim_{t\to t_i^-} \gamma'(t).$$
We then choose $\theta^\pm(t_i)$ to be the angle between the ray from $\gamma(0)$ to $\gamma(t)$ and the vector
$\gamma_\pm'(t)$.
(Equivalently, we could have defined $\theta^\pm(t_i)$ to be the angle between the ray from $\gamma(0)$ to 
$\gamma(t)$  and the geodesic segment $\gamma([t_i,t_{i\pm 1}))$.)
Notice that  $\theta(t)$ decreases smoothly on $(t_i,t_{i+1})$ for all $i$ and
that 
\begin{equation}
|\theta^+(t_i)-\theta^-(t_i)| \leq \phi_i
\label{dihedral bound}
\end{equation}
for all $i$.

If $t\ne t_i$ for any $i$,  then Lemma 4.4 in Epstein-Marden-Markovic \cite{EMM1} gives that
\begin{equation}
s'(t) = \cos(\theta(t)) \ \ \ {\rm and}\ \ \
\theta'(t) = -\frac{\sin(\theta(t))}{\tanh(s(t))} < -\sin(\theta(t)).
\label{diffeq}
\end{equation}

\subsection{The hill function of Epstein and Jerrard}

A key tool in Epstein and Jerrard's  work  is the following {\em hill function}  
$$h:\R\to (0,\pi) \text{ given by } h(x)= \cos^{-1}(\tanh(x)).$$
The  defining features of the hill function are that
$$h'(x)= -\sech(x) = -\sin(h(x))\ \ \ {\rm and}\ \ \  h(0)=\frac{\pi}{2}.$$
In particular, $h$ is a decreasing homeomorphism. 

For fixed $L>0$, we consider solutions to the equation
$$h'(x) = \frac{h(x) - h(x-L)}{L}.$$
Geometrically, we are finding the point on the graph of $h$ such that the tangent line at $(x,h(x))$
intersects the graph at the point $(x- L,h(x-L))$ (see Figure \ref{hillfig}). 
We will show that there is a unique solution $x=c(L)$ and that $c(L)\in (0,L)$. 

Given $x \in \R$, the tangent line at $(x,h(x))$ to the graph of $h$ intersects the graph in  two points $(x,h(x))$
and $(f(x),h(f(x))$ (except at $x = 0$ where the points are equal). The function $f$ is continuously differentiable and odd.
We define $A(x) = x-f(x)$, so $A$ is also continuously differentiable and odd.  Since $A$ is odd,
to show that $A$ is strictly  increasing, it suffices to show that it is strictly increasing on $[0, \infty)$.  
Suppose that  $0 \leq x_1 < x_2$, and that $T_1$ and $ T_2$ are the tangent lines to $h$ at $x_1$ and $x_2$.
Since $h$ is convex on $[0,\infty)$,  $T_1 \cap T_2=(x_0, y_0)$ lies  below the graph of $h$ and $x_1 < x_0 < x_2$.
Thus $T_2$ intersects the graph of $h$ to the left of the point of intersection of $T_1$ with the graph of $h$. 
Therefore, \hbox{$f(x_2) < f(x_1)\le f(0)=0$} and $f$ is decreasing. It follows that $A(x) = x-f(x)$ is increasing
and that $A(x)>x$ for all $x\in (0,\infty)$.
The function $c$ is the inverse of $A$, so $c$ is also continuous differentiable and strictly increasing. 
Since $A(x) > x$ for $x > 0$, $c(L) \in (0,L)$.

Let
$$\Theta(L) = h(c(L))\ \ \ {\rm and}\ \ \ G(L) = h(c(L)-L) - h(c(L)) = -Lh'(c(L)).$$
To show $G$ is monotonic, we define $B(x) = h(f(x))-h(x)$, the difference of the heights of the intersection points of the
tangent line at $(x,h(x))$ with the graph of $h$. As $h$ and $f$ are both strictly decreasing continuous functions, $B$ is 
strictly  increasing and continuous. Since $G(L) = B(c(L))$, $G$ is a strictly increasing continuous function.

We note that 
$$\Theta(L) + G(L) = h(c(L)-L) < \pi.$$
The following lemma is the key estimate in the proof of Theorem \ref{ej}. 

\begin{lemma}
If $\gamma:[0,\infty) \to \Ht$ is  piecewise geodesic, $L>0$  and 
$$||\gamma||_L\le G(L),$$
then
$$\theta^+(t) \leq \Theta(L) + G(L) <\pi$$
for all $t>0$.
\label{hill}
\end{lemma}

\begin{proof}
We define  maps $P^\pm:(0,\infty)\to \R^2$ which are continuous except at the bending
points $\{t_i\}$ and
whose image lies on the graph of $h$. Since $h$ is a homeomorphism onto
$[0,\pi]$, given $t\in (0,\infty)$, we can find a unique $g^\pm(t)\in\R$, such that
$$h(g^\pm(t)) = \theta^\pm(t).$$ 
We then define
$$P^\pm(t) =(P^\pm_1(t),P^\pm_2(t))= (g^\pm(t), h(g^\pm(t)))=(g^\pm(t),\theta^\pm(t)).$$ 
Note that the functions $P^+$ and $P^-$ agree except at the bending points.
In the intervals, we denote the common functions by $P(t)$, $g(t)$, and  $\theta(t)$.

Notice that as one moves along the geodesic ray $\gamma$, the functions
$\theta^\pm(t)$ decrease on each interval $(t_i,t_{i+1})$ and have vertical jump  equal to $\psi_i = \theta^+(t_i)-\theta^-(t_i)$
at  each $t_i$.  By equation \ref{dihedral bound} we have 
$$|\psi_i| = |\theta^+(t_i)-\theta^-(t_i)| \leq \phi_i.$$
Correspondingly, the point $P^\pm(t)$  move along the graph of $h$ by sliding rightward (and downward)
along  $(t_i,t_{i+1})$ and jumping vertically, either upwards or downwards, by $\psi_i$ at $t_i$,
see Figure \ref{hillfig}. 

We  argue by contradiction. Let $c = c(L)$, $G  = G(L)$, and $\Theta = \Theta(L)$.
Suppose there exists $T>0$ so that $\theta^+(T)>\Theta+G$. Let
$$s_0=\sup\{ s\in(0,T]\ |\ \theta^-(s)\le \Theta\}.$$
Notice that if $s_0=T$, then, since $|\theta^+(s_0)-\theta^-(s_0)|<G$,
$$\theta^+(T)\le \theta^-(T)+G\le \Theta+G$$
which would be a contradiction.

Also notice that $s_0=t_i$ for some $i$, since otherwise $\theta^- $ is continuous
and non-increasing at $s_0$, which would contradict the choice of $s_0$.

If $T-s_0<L$, then since $\theta$ can only increase at the bending points and the total
bending in the region $[s_0,T]$ is at most $G$, again
$$\theta^+(T)\le \theta^-(s_0)+G\le \Theta+G$$
which is a contradiction.

So, we may assume that $T-s_0\ge L$. We will use the assumption that
$\theta^-(t)>\Theta$ on \hbox{$(s_0,s_0+L]$} to arrive at a contradiction and
complete the proof of the lemma.

\medskip

We  show that under our hypotheses,
$P(T)$ cannot lie to the left of $(c(L)-L,h(c(L)-L))$.

The key observation in the proof is that 
$$h'(g(t))g'(t) = \theta'(t) < -\sin(\theta(t)) = - \sin(h(g(t))) = h'(g(t))$$
where the middle inequality follows from equation (\ref{diffeq}).
Since $h'(g(t))<0$, we conclude that $g'(t) >1$ for all $t\in(t_i,t_{i+1}).$
Therefore,
\begin{equation}
g(t_{i+1}) -g(t_i) =g^-(t_{i+1})-g^+(t_i) > t_{i+1} - t_i
\label{gmvt}
\end{equation}
for all $i$.

Let $\{s_0 = t_j, t_{j+1},\ldots, t_{j+m}\}$ be the bending points in the interval \hbox{$[s_0,s_0+L)$.}
For convenience, we redefine $t_{j+m+1}=s_0+L$.
Since  $||\gamma||_L\le G$, the total vertical jump in the region $[s_0,s_0+L)$ is at most $G$, i.e.
$$\sum_{i=j}^{j+m}| \theta^+(t_i)-\theta^-(t_i)| \le G,$$

Since $\theta^+$ is non-increasing on each interval $(t_i,t_{i+1})$ and
$\theta^-(s_0)\le\Theta$, it follows that
$$\theta^+(t)\le \Theta+G$$
for all $t \in [s_0,s_0+L)$.

\begin{figure}[htbp] 
   \centering
   \includegraphics[width=5in]{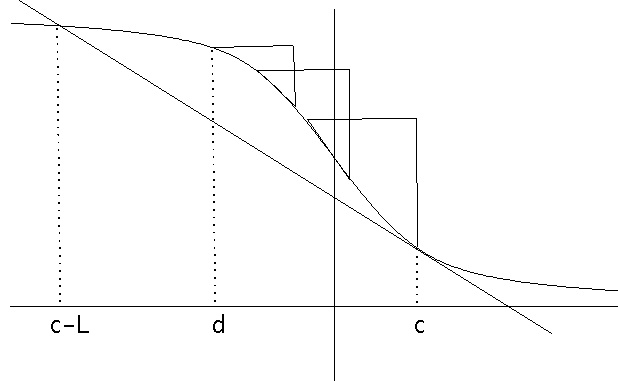} 
   \caption{Jumps and slides on the graph of $h$}
   \label{hillfig}
\end{figure}

Let 
$$d=\min\{ g^+(t)\ | \ t\in [s_0,s_0 +L)\}.$$
Notice that as $g^+$ is non-decreasing on $(t_i,t_{i+1})$
for all $i$, there exists a largest $k\in \{j,\ldots,j+m\}$ so that \hbox{$g^+(t_k)=d$}. 
We further note that $d \in[c-L,c]$
since $\theta^+(t)\in [\Theta, \Theta+G]$
for all \hbox{$t\in [s_0,s_0+L)$}. We break the proof into two cases.

\medskip\noindent
{\bf Case I:  $d \in[-c,c]$:} If $d \in[-c,c]$ then $g^+([s_0,s_0+L])\subseteq [-c,c]$.
Since $\theta^-(t) \geq \Theta$ on \hbox{$(s_0,s_0 + L]$}, we have $g^-((s_0,s_0+L])\subseteq [-c,c]$.
Notice that, since \hbox{$h'(x)=-\sin(h(x))$} and $h$ is decreasing,
if $x\in [-c,c]$, then
$$h'(x)\le h'(c)=-{G\over L}.$$
Therefore, applying (\ref{gmvt}), we see that
$$\theta^-(t_{i+1})-\theta^+(t_{i})\le h'(c)(g^-(t_{i+1})-g^+(t_{i}))=
-{G\over L}(g^-(t_{i+1})-g^+(t_{i})) \le -{G\over L}(t_{i+1}-t_{i})$$ 
for all $i=j,\ldots,j+m$.
Thus,
\begin{eqnarray*}
\theta^-(s_0+L)- \theta^-(s_0)  &=  & \left(\sum_{i=j}^{j+m} \theta^+(t_i)-\theta^-(t_i)\right) +
\left( \sum_{i=j}^{j+m} \theta^-(t_{i+1})-\theta^+(t_{i}) \right) \\
& \le &\left(\sum_{i=j}^{j+m} |\theta^+(t_i)-\theta^-(t_i)|\right)-\left(\sum_{i=1}^{j+m} \frac{G}{L}(t_{i+1}-t_i)\right) \\
& \le &  G - \frac{G}{L}\sum_{i=1}^{j+m} (t_{i+1}-t_i) = 0\ \ \ \ \ \ \ \ \ \ \ \ \ \ \ \ \ \ \ \ \ \ \ \ \ \ \ \ \ \ \ \ \\
\end{eqnarray*}
This implies that  $\theta^{-}(s_0+L) \le \Theta$, which  contradicts the choice of $s_0$.

\medskip\noindent
{\bf Case II: $d \in[c-L,,-c)$:} If $d\in [c-L,-c)$, then
$$|h'(g(t))|\ge |h'(d)|$$
for all $t\in [s_0,s_0+L]$. So,
\begin{equation}
(\theta^+(t_i)-\theta^-(t_{i+1})) \geq |h'(d)|(g^-(t_{i+1})-g^+(t_i)) \geq |h'(d)|(t_{i+1}-t_i)
\label{tmvt}
\end{equation}
for all $i=j,\ldots,j+m.$
It follows that 
$$\sum_{i=i}^{k-1}\left(\theta^+(t_i)-\theta^-(t_{i+1})\right) \ge |h'(d)| (t_k-s_0).$$
Thus, since $\theta^+(t_k)=h(d)$ and $\theta^-(t_j) \leq \Theta$,
$$\sum_{i=j}^{k}\left( \theta^+(t_i)-\theta^-(t_{i})\right)\ge (h(d)-\Theta)+|h'(d)|(t_k-s_0)$$
and so, since the total jump  on the interval $[s_0,s_0+L)$ is at most $G$,
$$\sum_{i=k+1}^{j+m} \theta^+(t_i)-\theta^-(t_i)\le G- (h(d)-\Theta)- |h'(d)|(t_k-s_0)
=h(c-L)-h(d)-|h'(d)|(t_k-s_0).$$

Since $g^+(t_k) = d$,
$$g^-(s_0+L)=d+\left(\sum_{i=k}^{j+m} g^-(t_{i+1})-g^+(t_i)\right) 
-\left(\sum_{i=k+1}^{j+m} g^-(t_i)-g^+(t_i)\right).$$
Applying inequalities  (\ref{gmvt}) and (\ref{tmvt}), we see that
$$g^-(s_0+L)> d+\left(\sum_{i=k}^{j+m}t_{i+1}-t_i\right)-\frac{1}{|h'(d)|}\left(\sum_{i=k+1}^{j+m} \theta^+(t_i)-\theta^-(t_i)\right)$$
$$>d+(s_0+L-t_k)-\frac{1}{|h'(d)|}\left(h(c-L)-h(d)- |h'(d)|(t_k-s_0)\right)$$
$$=d+L-\left({h(c-L)-h(d)\over |h'(d)|}\right).$$

Taking the tangent line at $d$ we note that, since $h'$ is negative and decreasing on the interval $[c-L,d]$, we have
$$h(c-L) \leq h(d) + h'(d)(c-L-d)$$
which implies that
$$ \frac{1}{h'(d)}(h(c-L)- h(d)) \geq  c-L-d.$$
Therefore,
$$g^-(s_0+L) > d+L +\frac{1}{h'(d)}(h(c-L)-h(d)) \geq c,$$
so, $\theta^-(s_0+L) \leq \Theta$ contradicting the definition of $s_0$.
This final contradiction completes the proof.
\end{proof}

As a nearly immediate corollary, we obtain an embeddedness criterion for piecewise geodesics.

\begin{corollary}
\label{embeddedgeodesic}
If $\gamma:[0,\infty)\to \Ht$ is a piecewise geodesic, 
and $||\gamma||_L \leq G(L)$ for some $L>0$, then $\gamma$ is an embedding. 
\end{corollary}

\begin{proof}
Notice that if the corollary fails, then there exists a piecewise geodesic  ray \hbox{$\gamma:[0,\infty)\to\Ht$} such that 
$||\gamma||_L\le G(L)$ and $\gamma(0)=\gamma(b)$
for some $b>0$. (Since if $\gamma(p)=\gamma(q)$ for some $0\le p<q$, we can instead
consider the piecewise geodesic ray $\gamma_1:[0,\infty)\to\Ht$ where $\gamma_1(t)=\gamma(t-p)$.)
There must exist $t_i\in (0,b)$ so that $\gamma$ is geodesic on $[t_i,b]$.
Then, $\theta^+(t) = \pi$ on $(t_i,b)$, contradicting Lemma \ref{hill} above.
\end{proof}

If $\mu$ is a finite-leaved measured lamination on $\Hp$ and $\alpha:[0,\infty)$ is any geodesic ray in $\Ht$,
then $\gamma=P_\mu\circ \alpha$ is a piecewise geodesic and $||\gamma||_L\le ||\mu||_L$. Since any
two points in $\Ht$ can be joined by a geodesic ray, we immediately obtain an embeddedness criterion
for pleated planes.

\begin{corollary}
If $\mu$ is a finite-leaved measured lamination on $\Hp$ and \hbox{$||\mu||_L\le G(L)$} for some $L>0$,
then $P_\mu:\Hp\to\Ht$ is an embedding.
\end{corollary}

\subsection{Uniformly bilipschitz embeddings}
We next prove that if \hbox{$\gamma:\R\to\Ht$} is a piecewise geodesic and $||\gamma||_L<G(L)$, then
$\gamma$ is uniformly bilipschitz.
We note that since $\gamma$ is 1-Lipschitz, we only have to prove a lower bound.
This will  immediately imply that if $\mu$ is a finite-leaved lamination on $\Hp$ and $||\mu_L||<G(L)$,
then $P_\mu$ is a  $K$-bilipschitz embedding.

\begin{proposition}
If $\gamma: \R \rightarrow \Ht$ is a piecewise geodesic such that 
$$||\gamma||_L  < G(L),$$ 
then $\gamma$ is $K$-bilipschitz where $K$  depends only on $L$ and $||\gamma||_L$.
\label{quasigeodesic}
\end{proposition}

\begin{proof} We first set our notation.
We may assume, without loss of generality,
that $0$ is not a bending point of $\gamma$. Let $t_0=0$ and assume that the bending points in $(0,\infty)$
are indexed by an interval  of  positive integers beginning with $1$ and the ending points in $(-\infty,0)$
are indexed by an interval of  negative integers ending with $-1$. Let $\phi_i$ be the bending angle
of $\gamma$ at $t_i$.

The following lemma will allow us to reduce to the planar setting.

\begin{lemma}
\label{planar variation}
There exists an embedded piecewise geodesic $\alpha:\R \rightarrow \Hp$ 
with the same bending points as $\gamma$
such that  
\begin{enumerate}
\item
if the bending angle of $\alpha$ at a bending point $t_i$ is given by $\phi_i'$, 
then $\phi_i'\le\phi_i$,
\item
$d(\alpha(0),\alpha(t)) = d(\gamma(0),\gamma(t))$ for all $t$, and
\item 
there exists a non-decreasing function $\Psi:\mathbb R\to (-\pi,\pi)$
such that if $t>0$, then $\Psi(t)$ is the angle between $\alpha([0,t_1])$ and the geodesic joining $\alpha(0)$ to $\alpha(t)$,
while if $t<0$, then $\Psi(t)$ is the angle between $\alpha([-t_1,0])$ and the geodesic joining $\alpha(0)$ to $\alpha(t)$.
\end{enumerate}
\end{lemma}


\begin{figure}[htbp] 
   \centering
   \includegraphics[width=4in]{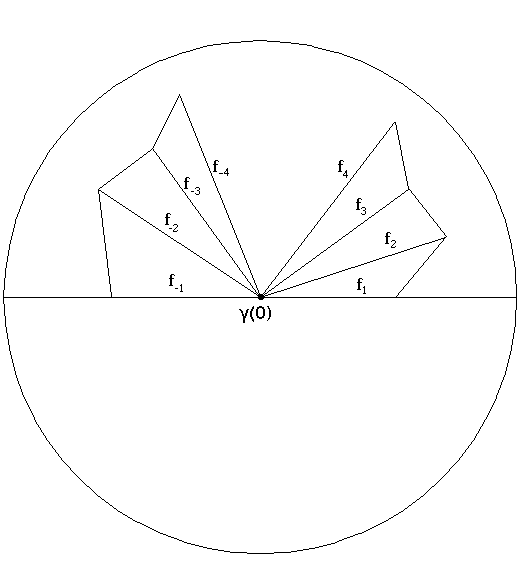} 
   \caption{The curve $\alpha$}
   \label{quasigeodesicfig}
\end{figure}

\begin{proof}
Let $f_i$ be the geodesic arc from $\gamma(0)$ to $\gamma(t_i)$ and let $T_i$ be the hyperbolic triangle with vertices
$\gamma(0)$, $\gamma(t_{i})$, and $\gamma(t_{i+1})$ and edges $f_i$, $\gamma([t_i,t_{i+1}])$ and $f_{i+1}$.
We construct $\alpha$ by first placing  an isometric copy of $T_0$ in $\Hp$, so that $f_1$ is counterclockwise from $f_0$.
We then iteratively place a copy of
$T_i$ adjacent to a copy of $T_{i-1}$(so that their interiors are disjoint) along the image of $f_i$ for all positive $t_i$.
We then place a copy of $T_{-1}$ in $\Hp$ so that $T_{-1}$ and $T_0$ intersect along the image of $\gamma(0)$, so
that the images of $f_1$ and $f_{-1}$ lie in a geodesic  and the image of $f_{-2}$ is clockwise from $f_{-1}$.
We then iteratively place a copy of $T_{-i-1}$ next to the copy of $T_{-i}$ for all negative $t_{-i}$
(see  Figure \ref{quasigeodesicfig}).

Let $\alpha:\mathbb R\to\Hp$ be the piecewise
geodesic traced out by the images of pieces of $\gamma$. Then $\alpha$ has the same bending points as $\gamma$
by construction. Moreover, since $d(\alpha(0),\alpha(t))$ is realized  in the isometric copy of $T_n$ when $t\in[t_n,t_{n+1}]$,
it is also immediate that $d(\alpha(0),\alpha(t)) = d(\gamma(0),\gamma(t))$ for all $t$.

We next check that the bending angle $\phi_i'$ of $\alpha$ at $t_i$ is at most $\psi_i$.
We consider the vectors $v^-_n = \gamma'_-(t_i)$ and $v_n^+ = \gamma'_+(t_i)$ at $\gamma(t_i)$. 
Then the  exterior angle $\phi_i$ is  the distance between $v_i^-$ and $v_i^+$ in the unit tangent sphere at $\gamma(t_i)$. 
The edge $f_n$ defines an axis in the unit sphere. The possibilities for gluing $T_n$ to $T_{n-1}$ are given by 
the one-parameter family of triangles obtained by rotating $T_i$ about $f_i$. It is then easy to see that the distance is
shortest when $T_i$ lies in the same plane as $T_{i-1}$ and has disjoint interior 
Therefore, $\phi_n'\le\phi_n$. Since 
$$||\alpha||_L\le||\gamma||_L<G(L),$$
Corollary \ref{embeddedgeodesic} implies that $\alpha$ is an embedding.

We can now define a continuous non-decreasing function $\Psi:\mathbb R\to\mathbb R$ so that $\Psi(0)=0$ and,
if $t>0$, then $\Psi(t)$ is the angle, modulo $2\pi$, between $\alpha([0,t_1])$ and the geodesic joining $\alpha(0)$ to $\alpha(t)$,
while if $t<0$, then $\Psi(t)$ is the angle between $\alpha([-t_1,0])$ and the geodesic joining $\alpha(0)$ to $\alpha(t)$.

We next show that $\Psi (t)< \pi$ for all $t>0$. If not, then $\gamma$ intersects the line $g_0$ containing
$\alpha([0,t_1])$. Suppose that $\alpha(b)\in g_0$ for some $b>0$. Then, consider the piecewise geodesic
$\hat\alpha$ which first traces $\alpha([0,b])$ backwards and then continues along $g_0$ forever.
Notice that $\hat\alpha$ is not an embedding. However,
$$||\hat\alpha||_L\le ||\alpha||_L<G(L),$$
so Corollary \ref{embeddedgeodesic} implies that $\hat\alpha$ is an embedding,
which is a contradiction. Similarly, $\Psi(t)> -\pi$ for all $t<0$. This completes the proof of (3).
 \end{proof}
 
We notice that it suffices to show that there exists $K$ depending only on $L$ and $||\gamma||_L$, so
that 
$$s(t)=d(\gamma(0),\gamma(t))=d(\alpha(0),\alpha(t))\ge K|t| $$
for all $t\in\mathbb R$.
Since, if we suppose that  $\gamma:\mathbb R\to \Ht$ is any piecewise geodesic with $||\gamma||_L<G(L)$
and $r_1<r_2$, then we can consider the new piecewise geodesic $\gamma_{r_1}:\mathbb\R\to \Ht$ given
by $\gamma_{r_1}(t)=\gamma(t-r_1)$. Then $||\gamma_{r_1}||_L=||\gamma||_L$ and 
$$s_{r_1}(t)=d(\gamma_{r_1}(t),\gamma_{r_1}(0))\ge K |t|.$$
It follows that
$$d(\gamma(r_1),\gamma(r_2))=s_{r_1}(r_2-r_1)\ge K|r_2-r_1|$$
Since $\gamma$ is 1-Lipschitz by definition, it follows immediately that $\gamma$ is a \hbox{$K$-bilipschitz} embedding.

\medskip

Since $\Psi$ is monotone and bounded we may define
$$\Psi_{+\infty}=\lim_{t\to\infty}\Psi(t)\ \ \ {\rm and}\ \ \ \ \Psi_{-\infty}=\lim_{t\to -\infty} \Psi(t).$$

We now show that $\alpha$ is proper. The basic idea is that, since $\Psi$ is monotonic, then $\alpha([0,\infty))$ can only 
accumulate on the geodesic ray $\vec r_+$ emanating from $\alpha(0)$ and making angle $\Psi_+$ with $\alpha([0,t_1])$.
If it accumulate at $q$, then there must be infinitely many segments of $\alpha$ running nearly parallel to
$\vec r_+$ and accumulating at some point $q$ on $\vec r$. However, by Lemma \ref{hill},  
no segment of $\alpha$ can be pointing nearly  straight back to $\alpha(0)$, so the total length of
these segments which are  ``pointing towards''  $\alpha(0)$ is finite. This will allow us to arrive at a contradiction.

If $\alpha$ is not proper, then either  $\alpha|_{[0,\infty)}$ or $\alpha|_{(-\infty,0]}$ is not proper. 
We may assume ray $\alpha|_{[0,\infty)}$ is not proper. We recall that if $t$ is not a bending point,
then $\theta(t)$ is the angle between $\alpha'(t)$ and the geodesic segment joining $\alpha(0)$ to $\alpha(t)$.
Lemma \ref{hill} implies that 
$$\theta(t)\le \Theta_0 = \Theta(L) + G(L) < \pi$$
for all $t$.
Since $\alpha|_{[0,\infty)}$ is not proper,  there is an accumulation point  $q$  of
$\alpha|_{[0,\infty)}$ on the ray $\vec r_+$ emanating from $\alpha(0)$  which makes an angle 
$\Phi_{+\infty}$ with $\alpha([0,t_1])$.

We may work in the disk model and assume that $\alpha(0)=0$ and $\alpha([0,t_1])$ lies in the positive
real axis. If $\epsilon>0$ is small enough, we can consider the region given in  hyperbolic polar coordinates
$(\bar r,\bar \theta)$ by
$$B_\epsilon = [r(q)-\epsilon, r(q)+\epsilon] \times [\theta(q)-\epsilon, \theta(q)]\subset \D^2.$$
On $B_\epsilon$ we consider the taxicab metric, given by $d_T((r_1,\theta_1),(r_2,\theta_2))= |r_1-r_2| + |\theta_1-\theta_2|$. We notice that that $d_T$ on $B_\epsilon$ is bilipschitz to the hyperbolic metric.
If $J=\alpha^{-1}(B_\epsilon)$, then $J$ is a countable collection of disjoint arcs.  Notice that
$\alpha(J)=\alpha([0,\infty))\cap B_\epsilon$.

Since $\Psi$ is monotonic, the $\bar\theta$ coordinate of $\alpha$ is monotonic, so  
the total length  of $\alpha(J)$ in the $\bar\theta$ direction  is bounded above by $\epsilon$. 
Also the signed length  of $\alpha(J)$ in the $r$ direction is bounded above by $2\epsilon$.
Since $\theta(t)\le\Theta_0$, at all non-bending points,
the total length in the negative $r$-direction is bounded above by 
$\epsilon\tan(\Theta_0)$. 
Therefore, the total length in the positive $r$-direction is bounded above by $\epsilon + \epsilon\tan(\Theta_0)$. 
It follows that $\alpha(J)$ has finite length in the taxicab metric  on $B_\epsilon$.  
We choose $\bar t\in J$, so that \hbox{$\alpha(J\cap [\bar t,\infty))$} has length, in the taxicab metric,
less than $\epsilon/4$ and $d_{B_\epsilon}(\alpha(\bar t),q)<\epsilon/4$. Therefore,
$\alpha(J\cap [\bar t,\infty))\subset B_{\epsilon/2}(q)$ and $\overline{B_{\epsilon/2}(q))}\subset B_\epsilon$
(where $B_{\epsilon/2}(q)$ is the neighborhood of radius $\epsilon/2$ of $q$ in the taxicab metric on $B_\epsilon$).
It follows that $[\bar t,\infty)\subset J$, which contradicts the fact that
$\alpha([\bar t,\infty))$ has infinite length. Therefore, $\alpha$ must be proper.

Since $\alpha$ is proper and $\Psi$ is monotone, $\alpha$ has two unique limit points $\xi^-$ and $\xi^+ $ in  $\Sp^1$
which are endpoints of the geodesic rays from $\alpha(0)$ which make angles
$\Psi_{-\infty}$ and $\Psi_{+\infty}$ with $\alpha([t_{-1},t_1])$. Thus, since $\alpha$ is embedded,
$$\Psi_{+\infty} - \Psi_{+\infty} \leq \pi.$$
Let 
$$B=\frac{G(L)-||\mu||_L}{2}$$
We further observe that 
$$\Psi_{+\infty} -\Psi_{-\infty} \le \pi - B$$ 
If not, we construct a new piecewise geodesic $\alpha_1:\mathbb R\to \Hp$ which
has a bend of angle $\frac{3(G(L)-||\mu||_L)}{4}$ at $0$. One then checks that
$$||\alpha_1||_L\le ||\alpha||_L+3/4\left( G(L)-||\mu||_L\right)<G(L)$$ 
but $\alpha_1$ is not an embedding, which would contradict Corollary
\ref{embeddedgeodesic}.

Let $g$ be the geodesic joining $\xi^-$ to $\xi^+$.
Since $\Psi_{+\infty}-\Psi_{-\infty}  \le \pi-B$, the
visual distance between $\xi^+$ and $\xi^-$,
as viewed from $\alpha(0)$ is at least $B$. It follows that there
exists $C$, depending only on $B$, so that $d(\alpha(0),g))\le C$.
In fact,  one may apply Theorem 7.9.1 in Beardon \cite{beardon} to check that we may 
choose
$$C = \cosh^{-1}\left(\frac{1}{\sin(B/2)}\right).$$ 
Notice that, by considering a reparameterization of $\alpha$, we can see that the visual distance
between $\xi^+$ and $\xi^-$ is at least $B$ as viewed from $\alpha(t)$ for any $t\in\mathbb R$,
and thus that $\alpha(t)$ lies within $C$ of $g$ for any $t\in\mathbb R$.

We next claim there exists $K>0$ such that if $p:\Hp\to g$ is orthogonal projection, then
$p\circ\alpha$ is a $1$-Lipschitz, $K$-bilipschitz orientation-preserving embedding. The fact
that $p\circ\alpha$ is 1-Lipschitz follows immediately from the fact that both $p$ and $\alpha$ are
1-Lipschitz. Let $\nu_0$ be the angle between the orthogonal geodesic $h_0$ to $g$ through $\alpha_0$
and the geodesic segment $\alpha([t_{-1},t_1])$ chosen so that $\nu_0>0$ if $\alpha(t_1)$ lies on the same side
of $h_0$ as $\xi^+$. Notice that  
$$\frac{B}{2}\le \nu_0 \le \pi-\frac{B}{2}$$
since otherwise $\Psi_{+\infty}-\Psi_{-\infty}\ge \pi-B$.
Therefore, the restriction of $p\circ\alpha$ to $[t_{-1},t_1]$ is an orientation-preserving embedding.
We let $v_0$ be a unit tangent vector at $\alpha(0)$ perpendicular to $g$. Then  
$$||p'(\alpha(0))(v)|| = \frac{1}{\cosh(d(\alpha(0),g))} \ge \frac{1}{\cosh(C)} = \sin(B/2)$$
As $\alpha'(0)$ makes an angle at most $B/2$ with $v$
$$||(p\circ\alpha)'(0)|| \ge \frac{\sin(B/2)}{\cosh(C)} = \sin^2(B/2) = \frac{1}{K}.$$
Again, by reparameterizing, we may check that if $t$ is a non-bending point, then $p\circ\alpha$ is an
orientation-preserving local homeomorphism at $t$ and that 
$$||(p\circ\alpha)'(t)||\ge \frac{1}{K}.$$
It follows that, for all $t$,
$$d(p(\gamma(0)), p(\gamma(t))\ge \frac{1}{K}t.$$
Therefore, since $p$ is 1-Lipschitz,
$$s(t) = d(\alpha(0),\alpha(t)) \geq d(p(\gamma(0)), p(\gamma(t))) \geq \frac{t}{K} $$
We observed earlier that this is enough to guarantee that 
$\gamma$ is  $K$-bilipschitz.
\end{proof}

As an immediate corollary, we obtain a version of Theorem \ref{ej} for finite-leaved laminations.

\begin{corollary}
\label{finite ej}
If $\mu$ is a finite-leaved measured lamination on $\Hp$ such that 
$$||\mu||_L < G(L),$$ 
then $P_\mu$ is a $K$-bilipschitz embedding, where $K$ depends 
only on $L$ and $||\mu||_L$.
\end{corollary}

\subsection{Proof of Theorem \ref{ej}}
Suppose that $\mu$ is a measured lamination on $\Hp$ with $||\mu||_L  < G(L)$. 
By Lemma 4.6 in Epstein-Marden-Markovic \cite{EMM2}, there exists a sequence $\{\mu_n\}$ of finite-leaved measured
laminations  which converges to $\mu$ such that $||\mu_n||_L = ||\mu||_L$ for all $n$.
Corollary \ref{finite ej} implies that each $P_{\mu_n}$ is a $K$-bilipschitz embedding
where $K$  depends only on $L$ and $||\mu||_L$.
The maps $\{P_{\mu_n}\}$ converges uniformly on compact sets to $P_\mu$ (see \cite[Theorem III.3.11.9]{EM87}),
so $P_\mu$ is also a $K$-bilipschitz embedding. Therefore, $P_\mu$ extends  continuously to
$\hat{P}_\mu: \Hp\cup \Sp^1_\infty \rightarrow \Ht\cup\Sp^2_\infty$ and $\hat P_\mu(\Sp^1)$ is a quasi-circle.
\eproof

\section{Complex earthquakes}

In this section, we use Theorem \ref{ej} to give improved bounds in results of Epstein-Marden-Markovic
which will lead to the improved bound obtained in our main result. We first obtain new bounds 
guaranteeing that complex earthquakes extend to homeomorphisms at infinity, see Corollaries \ref{eqembedding}
and \ref{eqembedding2}. Once we have done
so, we obtain a generalization of \cite[Theorem 4.14]{EMM1} which produces a family of conformally natural quasiconformal maps
associated to complex earthquakes with the same support $\mu$ which satisfy the bounds obtained in
Corollary \ref{eqembedding} or Corollary \ref{eqembedding2}.
Finally, we give a version of \cite[Theorem 4.3]{EMM2} which gives rise to a family of quasiregular maps
associated to all complex earthquakes with positive bending along $\mu$.

If $\mu$ is a measured lamination on $\Hp$,  we define $E_{\mu}: \Hp \rightarrow \Hp$ to be the earthquake map defined by 
fixing a component of the complement of $\mu$ and left-shearing all other components by an amount given by the measure on 
$\mu$.  An earthquake map is continuous  except on leaves of $\mu$ with discrete measure and extends to a homeomorphism
of $\Sp^1$.Therefore, any measured lamination $\lambda$ on $\Hp$ is mapped to a  well-defined measured lamination on $\Hp$
which we denote $E_\mu(\lambda)$. 

Given a measured lamination $\mu$ on $\Hp$ and $z=x+iy\in\mathbb C$, 
we define the {\em complex earthquake} 
$$\C E_{z}=P_{yE_{x\mu}}\circ E_{x\mu}: \Hp \rightarrow \Ht$$ 
to be the composition of earthquaking along $x\mu$ 
and then bending along the lamination $yE_{x\mu}(\mu)$. The sign of $y$ determines the direction of the bending. 
By linearity,
$$|| yE_{x\mu}(\mu)||_L  = |y|\ ||E_{x\mu}(\mu)||_L.$$
(See Epstein-Marden \cite[Chapter 3]{EM87} or Epstein-Marden-Markovic \cite[Section 3]{EMM1} for a detailed
discussion of complex earthquakes.)

The following estimate allows one to bound $||E_{x\mu}(\mu)||_L$.

\begin{theorem}{\rm (Epstein-Marden-Markovic \cite[Theorem 4.12]{EMM1})}
\label{earthquake bounds}
Let $\ell_1$ and $\ell_2$ be distinct leaves of a measured lamination  $\mu$ on $\Hp$.
Suppose that $\alpha$ is a closed geodesic segment with endpoints on $\ell_1$ and $\ell_2$ and let
$x=i(\alpha,\mu)$.
Let $\ell_1'$ and $\ell_2'$  be the images of  $\ell_1$ and $\ell_2$ under the earthquake  $E_\mu$. 
Then 
$$\sinh(d(\ell_1', \ell_2')) \leq e^x \sinh(d(\ell_1, \ell_2)) \ \ \  \mbox{\rm and } \ \ \ \ d(\ell_1', \ell_2') \leq e^{x/2}d(\ell_1, \ell_2).
$$
Furthermore,
$$\sinh(d(\ell_1 , \ell_2)) \leq e^x \sinh(d(\ell_1', \ell_2')) \ \ \  \mbox{\rm and } \ \ \ \ d(\ell_1, \ell_2) \leq e^{x/2}d(\ell_1', \ell_2').
$$
\end{theorem}

Motivated by this result, Epstein, Marden,  and Markovic define the function
$$f(L,x) = \min\left(Le^{|x|/2}, \sinh^{-1}(e^{|x|}\sinh(L))\right).$$ 

Corollary 4.13 in \cite{EMM1} generalizes to give:

\begin{corollary}
If $\mu$ is a measured lamination on $\Hp$, $z=x+iy\in\mathbb C$,  and $L>0$, 
then
$$||E_{x\mu}(\mu)||_{L}  \leq \left\lceil{\frac{f(L,x)}{L}}\right\rceil ||\mu||_L.$$
Furthermore, if  
$$|y| < \frac{G(L)}{\left\lceil{\frac{f(L,x)}{L}}\right\rceil ||\mu||_L},$$
then $\C E_z$ extends to an embedding of $\Sp^1$ into $\rs$.
\label{eqembedding}
\end{corollary}

We similarly define
$$g(L,x) = \max\left(Le^{-|x|/2}, \sinh^{-1}(e^{-|x|}\sinh(L))\right).$$
We will show later, see Lemma \ref{g is simple}, that if $2 \tanh(L) > L$ then \hbox{$g(L,x) = L e^{-|x|/2}$}.

Theorem \ref{earthquake bounds} and Theorem \ref{ej} combine to give the following:
 
\begin{corollary}
If $\mu$ is a measured lamination on $\Hp$, $z=x+iy\in\mathbb C$,  and $L>0$, then
 $$||E_{x\mu}(\mu)||_{g(L,x)}  \leq ||\mu||_L.$$
Furthermore, if 
$$|y| < \frac{G(g(L,x))}{||\mu||_L},$$
then $P_{yE_{x\mu}}$ is a bilipschitz embedding and $\C E_z$ extends to an
embedding of $\Sp^1$ into $\rs$.
\label{eqembedding2}
\end{corollary}

\noindent
{\em Proofs:}
The proofs of  Corollaries \ref{eqembedding} and \ref{eqembedding2} both follow the same outline as the proof
of \cite[Corollary 4.13]{EMM1}. 
Let $\mu$ be a measured lamination on $\Hp$, $z=x+iy\in\mathbb C$,  and $L>0$.

Suppose that  $A>0$ and  that $\alpha$ is an open geodesic arc in $\Hp$ of length $A$ which is
transverse to $E_{x\mu}(\mu)$. Theorem \ref{earthquake bounds} guarantees that one can choose
an open geodesic arc $\beta$ in $\Hp$ which intersects exactly the leaves of $\mu$ which correspond
to leaves of $E_{x\mu}$ which intersect $\alpha$ and has total length at most $f(A,x)$.
Therefore,
$$i(\alpha,E_{x\mu}(\mu))=i(\beta,\mu)\le ||\mu||_{f(A,x)},$$
so
\begin{equation}
||E_{x\mu}(\mu)||_A \leq ||\mu||_{f(A,x)}.
\label{norm bound}
\end{equation}

We begin with the proof of Corollary \ref{eqembedding2}. Inequality \ref{norm bound}
immediately implies that 
$$||E_{x\mu}(\mu)||_{g(L,x)} \leq ||\mu||_{f(g(L,x))}=||\mu||_L.$$
So, if
$$|y| < \frac{G(g(L,x))}{||\mu||_L},$$
then 
$$||y \ E_{x\mu}||_{g(L,x)} <  G(g(L,x)).$$ 
Theorem \ref{ej} then implies that  $P_{yE_{x\mu}}$ is a bilipschitz embedding which extends to an
embedding of $\Sp^1$ into $\rs$. Since $E_{x\mu}$ extends to  a homeomorphism of $\Sp^1$, it follows
that $\C E_z$ extends to an
embedding of $\Sp^1$ into $\rs$. 
This completes the proof of Corollary \ref{eqembedding2}.

We now turn to the proof of Corollary  \ref{eqembedding}.
We can divide a half open  geodesic arc in $\Hp$ of length $f(L,x)$
into  $\lceil f(L,x)/L\rceil$ half open  geodesic arcs of length less than or equal  to $L$, so
$$||E_{x\mu}(\mu)||_L \leq ||\mu||_{f(L,x)} \leq \left\lceil \frac{f(L,x)}{L}\right\rceil ||\mu||_L.$$
Therefore, if 
$$|y| <  \frac{G(L)}{\left\lceil{\frac{f(L,x)}{L}}\right\rceil ||\mu||_L},$$
then 
$$||y\ E_{x\mu}(\mu)||_L < G(L).$$
and we may again use Theorem \ref{ej} to complete the proof of Corollary \ref{eqembedding}.
\eproof

For all $L>0$, we define
$$Q(L,x) = \max\left( \frac{G(L)}{\left\lceil{\frac{f(L,x)}{L}}\right\rceil}, G(g(L,x))\right)$$
and
$${\mathcal T}^L_0 ={\rm int}( \left\{ x+iy \ | \ |y|  < Q(L,x)\right\}.$$

The following theorem is a direct generalization of Theorem 4.14 in Epstein-Marden-Markovic \cite{EMM1}.
In its proof, we simply replace their use of Corollary 4.13 in \cite{EMM1} with our Corollaries
\ref{eqembedding} and \ref{eqembedding2}.

\begin{theorem}
Suppose that $L>0$ and $\mu$ is a measured lamination on $\Hp$  such that $||\mu||_L= 1$.
Then, for $z\in {\mathcal T}_0^L$, 
\begin{enumerate}
\item
$\C E_z$ extends to an embedding $\phi_z:\Sp^1\to \rs$ which bounds a region $\Omega_z$. 
\item 
There is a  quasiconformal map
$\Phi_z:\D^2 \rightarrow \Omega_z$ with domain the unit disk  and quasiconformal dilatation $K_z$ bounded  by 
$$K_z \leq \frac{1 + |h(z)|}{1-|h(z)|}$$
where $h:{\mathcal T}_0 \rightarrow \D^2$ is a Riemann map taking $0$ to $0$.

Moreover, $\Phi_z\cup\phi_z:\D^2\cup \Sp^1\to \rs$ is continuous.

\item If $G$ is a group of M\"obius transformations preserving $\mu$, then 
$\Phi_z$ can be chosen so that there is a homomorphism $\rho_z:G \to G_z$ where 
$G_z$ is also a group of M\"obius transformations and 
$$\Phi_z \circ g = \rho_z(g)\circ \Phi_z$$
for all $g\in G$.
\end{enumerate}
\label{holo motion}
\end{theorem}

Epstein, Marden and Markovic \cite{EMM2} introduce the theory of complex angle scaling maps
and use them to produce a family of quasiregular mappings indexed by
$${\mathcal S}^L = {\rm int}\left\{ x+iy\in\mathbb C \ | \ y > -\frac{0.73}{f(1,x)}\right\}$$
so that if $|{\rm Im}(t)|< \frac{0.73}{f(1,x)}$, then $\Phi_t$ is quasiconformal.
(See also the discussion in \cite[Section 3.4]{BC13}.)

We consider the enlarged region
$${\mathcal T}^L = {\rm int}\left\{ x+iy\in\mathbb C \ | \ y > -Q(x,L)\right\}.$$
Given Theorems \ref{ej} and \ref{holo motion}, their proof of Theorem 4.3 extends immediately to give:

\begin{theorem}{\rm (\cite[Theorem 4.13]{EMM2})}
Suppose that $L>0$, $\mu$ is a measured lamination on $\Hp$ with $||\mu||_L  = 1$,
$v_0>0$ and $t_0 = iv_0 \in {\mathcal T}^L_0$. If $t\in {\mathcal T}^L_0$, let
$\Omega_t$ be the the image of $\D^2$  under the map $\Phi_{t}$ given by Theorem \ref{holo motion}.
Then there exists a continuous map
$\Psi:\U \times \Omega_{t_0} \rightarrow \rs$, where $\U$ is the upper half-plane, such that
\begin{enumerate}
\item $\Psi_{t_0} = id$.
\item For each $z \in \Omega_{t_0}$, $\Psi(t,z)$ depends holomorphically on $t$.
\item For each $t \in {\mathcal T}^L_0$, $\Psi_t$ can be continuously extended to $\partial \Omega_{t_0}$ such that 
$$\left.\Psi_t \circ \Phi_{t_0}\right|_{\Sp^1} = \left.\Phi_t\right|_{\Sp^1}.$$
In particular $\Psi_0: \partial\Omega_{t_0} \rightarrow \Sp^1$ and $\Phi_{t_0}: \Sp^1 \rightarrow \partial \Omega_0$ are inverse homeomorphisms.
\item If $t \in {\mathcal T}^L_0$ and ${\rm Im}(t)>0$, then  $\Psi_t$ is injective and
\hbox{$\Psi_t(\Omega_0)  =\Phi_t(\D^2) = \Omega_t$.}
\item  If $t = u+iv$ and $v>0$, then $\Psi_t$ is locally injective $K_t$-quasiregular mapping where
$$K_t - \frac{1+|\kappa(t)|}{1-|\kappa(t)|}, \qquad |\kappa(t)| = \frac{\sqrt{u^2 +(v-v_0)^2}}{\sqrt{u^2 +(v+v_0)^2}}$$
\item
If $G$ is a group of M\"obius transformations preserving $\Omega_0$,
then there is a homomorphism $\rho_t:G\to  G_t$ where 
$G_t$ is also a group of M\"obius transformations, such that 
$$\Psi_t \circ g = \rho_t(g)\circ \Psi_t$$
for all $g\in G$.

\end{enumerate}
\label{analytic continuation}
\end{theorem}

\section{Quasiconfomal bounds}

One can now readily adapt the techniques of proof of  Epstein-Marden-Markovic \cite[Theorem 6.11]{EMM2}
to establish:

\begin{theorem}
If $\Omega$ is a simply connected hyperbolic domain in $\rs$ and $L>0$, then 
there is a conformally natural $K$-quasiconformal map $f:\Omega\to\chb$ which extends to the identity on 
$\partial \Omega\subset\rs$ such that
$$\log(K) \leq d_{{\mathcal T}^L}(ic_1(L),0)$$
where $d_{{\mathcal T}^L}$ is the Poincar\'e metric on the domain ${\mathcal T}^L$ and
\hbox{$c_1(L) = 2 \cos^{-1}\left(-\sinh\left(\frac{L}{2}\right)\right)$.}
\label{qc bound}
\end{theorem}

We offer a brief sketch of the proof in order to indicate where our new bounds, 
as given in Theorems \ref{bendboundbetter}, \ref{holo motion} and \ref{analytic continuation},
are used in the argument.

We recall that universal Teichm\"uller space ${\mathcal U}$ is the space of 
quasisymmetric homeomorphisms of the unit ciricle $\Sp^1$, modulo the action of M\"obius transformations
by post-composition (see, for example, Ahlfors \cite[Chapter VI]{ahlfors}.  
The Teichm\"uller metric on the space $\mathcal{U}$ is defined by
$$d_{\mathcal U}(f,g) =   \log \inf K(\hat{f}^{-1}\circ \hat{g})$$ 
where the infimum is over all quasiconformal extensions $\hat f$ and $\hat g$  of $f$ and $g$ to maps from the unit disk to
itself and  $K(\hat{f}^{-1}\circ \hat{g})$ is the quasiconformal dilatation of $\hat{f}^{-1}\circ\hat g$.
If $\Gamma$ is a  group of conformal automorphisms of $\D^2$, we define ${\mathcal U}(\Gamma) \subseteq {\mathcal U}$ to be
the quasisymmetric  homeomorphisms which conjugate the action of $\Gamma$ to the action of an isomorphic group
of conformal automorphisms. The Teichm\"uller metric on $\mathcal{U}(\Gamma)$ is defined similarly by considering extensions
which conjugate $\Gamma$ to a group of conformal automorphisms.

Let $g: \D^2 \rightarrow \rs$ be a locally injective quasiregular map, i.e. $g = h\circ f$ where $f$ is a quasiconformal
homeomorphism and $h$ is locally injective and holomorphic on the image of $f$. We may define a complex structure
$C_g$ on $\D^2$ by pulling back the complex structure on $\rs$ via $g$. The identity map defines a quasiconformal 
homeomorphism $\hat{g}: \D^2 \rightarrow C_g$. We then uniformize $C_g$ by a conformal map $R: C_g \rightarrow \D^2$ 
and consider the quasiconformal map $R\circ \hat{g}: \D^2 \rightarrow \D^2$. This map extends to the boundary to 
give a quasisymmetric map $qs(g):\Sp^1 \rightarrow \Sp^1$.

Choose $\mu$ so that $\chb=P_{c\mu}(\D^2)$ where $||\mu||_L=1$ and $c>0$.
We use Theorem \ref{holo motion} to define a map 
$$F: {\mathcal T}_0^L \rightarrow {\mathcal U}(\Gamma),$$
where $\Gamma$ is the group of conformal automorphisms of $\Hp$ preserving $\mu$.  
If $t \in {\mathcal T}_0^L$, let  
$$F(t) = qs(\Phi_t).$$

Similarly, we may use Theorem \ref{analytic continuation}, with some choice of $t_0=iv_0\in{\mathcal{T}}^L_0$, to define a map
$$G: \U\rightarrow {\mathcal U}(\Gamma)$$
by letting  
$$G(t) = qs(\Psi_t\circ\Phi_{t_0}).$$

If $t$ lies in the intersection of the domains of $F$ and $G$, then even
though $\Phi_t$ and $\Psi_t\circ\Phi_{t_0}$ need not agree on $\D^2$, Theorem \ref{analytic continuation}
implies that they have the same boundary values and quasi-disk image $\Omega_t$.
Therefore $F$ and $G$ agree on the overlap $\mathcal{T}^L_0\cap\U$ of their domains.
We may combine the functions to obtain
a well-defined function 
$$\bar F:{\mathcal T}^L \rightarrow {\mathcal U}(\Gamma).$$
Epstein, Marden, Markovic further show that $\bar F$ is holomorphic  (see \cite[Theorem 6.5 and Proposition 6.9]{EMM2}).

The Kobayashi metric on a complex manifold $M$ is  defined to be the largest metric on $M$ with the property
that for any holomorphic map $f:\D^2 \rightarrow M$, $f$ is 1-Lipschitz with respect the hyperbolic metric on $\D^2$. 
Therefore, holomorphic maps between complex manifolds are 1-Lipschitz with respect to  their Kobayashi metrics.
The Teichm\"uller metric agrees with the Kobayashi metric  on $\mathcal{U}$ and $\mathcal{U}(\Gamma)$ 
(see \cite[Chapter 7]{GarLak}). Morever, the Poincar\'e metric on any simply connected domain,
in particular $\mathcal{T}^L$, agrees with its Kobayashi metric.
 It follows then that for any $t \in {\mathcal T}^L$,
$$d_{\mathcal{U}(\Gamma)}(\bar F(t), \bar F(0)) \leq d_{{\mathcal T}^L} (t,0).$$ 
Theorem \ref{bendboundbetter} implies that
$$c  \leq c_1(L) = 2 \cos^{-1}\left(-\sinh\left(\frac{L}{2}\right)\right)$$
so
$$d_{\mathcal{U}(\Gamma)}(\bar F(ic), \bar F(0)) \leq d_{{\mathcal T}^L} (ic,0) \leq d_{{\mathcal T}^L} (ic_1(L),0).$$

Since $\C E_{ic}=P_{c\mu}$ and $\Omega=\Omega_{ic}$ is simply connected, the map $g_{ic}=\Psi_{ic}\circ\Phi_{t_0}$
is a conformally natural quasiconformal mapping with image $\Omega$.
Moreover, $P_{c\mu}\circ g_{ic}^{-1}:\Omega\to\chb$ extends to the identity on $\partial \Omega=\partial\chb$.
(For more details, see the discussion in the proofs of \cite[Theorem 6.11]{EMM2} or  \cite[Theorem 1.1]{BC13}.)

We have that $\bar F(ic) = qs(g_{ic}) =\left.(R\circ g_{ic})\right|_{\Sp^1}$ where
$R:\Omega \rightarrow \D^2$ is a uniformization map.
Therefore, 
$$d_{\mathcal{U}(\Gamma)}(\bar F(ic),\bar F(0)) = d_{\mathcal{U}(\Gamma)}(\bar F(ic), Id) = \log \inf K(h)$$
where the infimum is taken over all  quasiconformal maps  from $ \D^2$ to $\D^2$
extending $\left.(R\circ g_{ic})\right|_{\Sp^1}$ and conjugating $\Gamma$ to a group of conformal automorphisms.
By basic compactness results for
families of quasiconformal maps, this infimal quasiconformal dilatation is achieved  by 
a quasiconformal map $h:\D^2\to \D^2$.
If $f: \Omega \rightarrow \D^2$ is given by
$f=h^{-1}\circ R$, then
$$K(f) = K(h)=d_{\mathcal{U}(\Gamma)}(\bar F(ic), \bar F(0))  \leq d_{{\mathcal T}^L} (ic_1(L),0).$$
Since $h$ and $R\circ g_{ic}$ are quasiconformal maps with the same extension to $\partial \Hp$,
they are boundedly homotopic (see, e.g., \cite[Lemma 5.10]{EMM2}). So, $f$ is boundedly homotopic to $g_{ic}^{-1}$.
Thus,  $P_{c\mu}\circ f:\Omega \rightarrow \chb$  is 
boundedly homotopic to $P_{c\mu}\circ g_{ic}^{-1}$. Since $P_{c\mu}\circ g_{ic}^{-1}$  extends to the identity on 
$\partial\Omega$, it follows that $P_{c\mu}\circ f$ also extends to the identity on $\partial\Omega$.
Therefore,  \hbox{$P_{c\mu}\circ f:\Omega \rightarrow \chb$}  is the desired conformally natural
$K$-quasiconformal map which extends to the identity on $\partial \Omega$
such that
$$\log(K) \leq d_{{\mathcal T}^L} (ic_1(L),0).$$ 
This completes the sketch of the proof of Theorem \ref{qc bound}.

\medskip\noindent
{\bf Remark:} Epstein, Marden and Markovic showed that if $\Omega$ is simply connected, then
a quasiconformal map between $\Omega$ and $\chb$ extends to the identity on $\partial \Omega$
if and only if it is boundedly homotopic to the nearest point retraction from $\Omega$ to $\chb$ (see \cite[Theorem 5.9]{EMM2}).

\section{Derivation of main theorem}

In order to complete the proof of our main theorem, Theorem \ref{main}, it suffices to show
that one can choose $L>0$ such that
$$d_{{\mathcal T}^L}(ic_1(L),0)<7.1695.$$
Motivated by computer calculations for various values of $L$, we choose \hbox{$L = 1.48$}.

First, we construct a polygonal approximation for the region $\mathcal{T}^L$ from within,  see figure  \ref{Tapprox}.
\begin{figure}[htbp] 
\centering
   \includegraphics[width=5in]{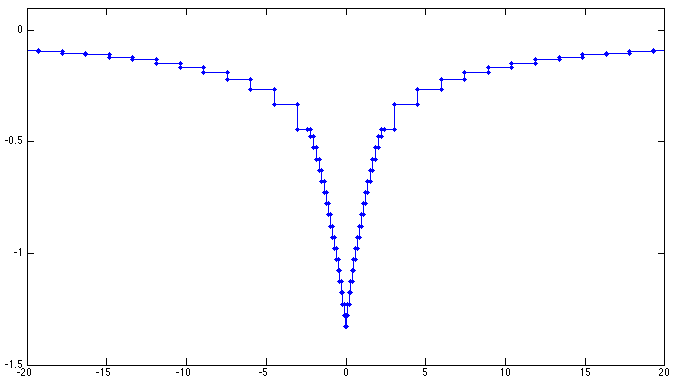} 
   \caption{Polygonal approximation of $\mathcal{T}^L$}
\label{Tapprox}
\end{figure}
The approximation is constructed using MATLAB's Symbolic Math Toolbox and variable precision arithmetic. 
Variable precision arithmetic allows us to compute vertex positions to arbitrary precision. 
In particular, we can deduce sign changes to find intervals containing intersection points.

We  build a step function $s(x) \leq Q(L,x)$ as follows; We recall that
$$Q(L,x) = \max\left( \frac{G(L)}{\left\lceil{\frac{f(L,x)}{L}}\right\rceil}, G(g(L,x))\right).$$
We first locate intervals  where $\frac{G(L)}{\left\lceil{\frac{f(L,x)}{L}}\right\rceil}$ and 
$G(g(L,x))$ intersect. For values where $\frac{G(L)}{\left\lceil{f(L,x)/L}\right\rceil}$ dominates, 
we bound $Q(L,x)$ by truncated decimal expansions (i.e. lower bounds) of values of 
$\frac{G(L)}{\left\lceil{f(L,x)/L}\right\rceil}$, which we compute using variable precision arithmetic.

For parts dominated by $G(g(L,x))$, we simplify our computation by using the following lemma.

\begin{lemma}
\label{g is simple}
Let $L_0 > 0$ be the unique positive solution to $2\tanh(L) = L$. 
If $L < L_0\approx 1.91501$,   then $g(L,x) = L e^{-|x|/2}$. 
\end{lemma}

\begin{proof} Recall that $$g(L,x) = \max\left(Le^{-|x|/2}, \sinh^{-1}(e^{-|x|}\sinh{L})\right).$$
Let $L < L_0$ and consider the function $j(x)=e^x \sinh(L e^{-x/2})$. 
It has a critical point precisely when
$$2\tanh(Le^{-x/2}) = Le^{-x/2}.$$ 
Since $L < L_0$, we have $Le^{-x/2} < L_0$  when $x\ge 0$, so $j$ has
no critical points in the interval $[0,\infty)$. Since $j'(0)=\sinh L-\frac{L}{2}  \cosh L >0$, 
$j$ is  increasing
on  the interval $[0, \infty)$.  Therefore,
$$j(x)=e^x \sinh(L e^{-x/2}) \geq \sinh(L)=j(0)$$
for all $x\ge 0$, so
$$Le^{-x/2} \geq \sinh^{-1}(e^{-x} \sinh(L))$$
for all $x\ge 0$.
Thus, $g(L,x) = L e^{-|x|/2}$ for all $x$.
\end{proof}

From our initial analysis of the hill function, we know that $G(t)$ is an increasing function on $t \in [0, \infty)$.
It follows that $G(g(L,x))$ is a decreasing function for $x \in [0, \infty)$. 
Therefore, we can approximate $G(g(L,x))$ by a step function from below.

To compute the values of $G(g(L,x))$, recall that \hbox{$G(t) = h(c(t)-t)-h(c(t))$}. 
The function $c(t)$ can be computed to arbitrary precision from the equation 
$$t \, h'(c(t)) = h(c(t)) - h(c(t)-t).$$
In particular, variable precision arithmetic can give us truncated decimal expansions of values of $G(g(L,x))$. 
We sample at a collection of points to obtain a step function where $G(g(L,x))$ dominates.

We use these computations to build $s(x) \leq Q(L,x)$ on some interval $[-a,a]$. Outside of that interval, we set $s(x) = 0$. The graph of $-s(x)$ gives us the boundary of a polygonal region contained in ${\mathcal T}^L$. 

Using the Schwarz-Christoffel mapping toolbox developed by Toby Driscoll \cite{SC}, the images of the points
$0$ and $2 \cos^{-1}\left(-\sinh\left(\frac{L}{2}\right)\right)i$ are computed under a Riemann mapping of the 
approximation of $\mathcal{T}^L$ to the upper half plane. Computing the hyperbolic distance between the 
images provides the result. The Schwarz-Christoffel mapping toolbox provides precision and error estimates. 
The error bounds are on the order of $10^{-5}.$

We found that the optimal bound is given when $L$ is approximately $1.48$. 
Using $L = 1.48$, the point 
$$B=c_1(L)i=2 \cos^{-1}\left(-\sinh\left(\frac{L}{2}\right)\right)i \approx 5.027888826784i$$
and  
$$e^{d_{{\mathcal T}^L}(ic_1(L),0)}\approx 7.16947.$$
A truncated version of the output provides the values of $G(L)$, $HPL(0)$,  and $HPL(B)$, 
where $HPL:\mathcal{T}^L\to \Hp$ is a Riemann mapping from $\mathcal{T}^L$ to the upper half-plane.
We also have \hbox{$H(L)=d_{\Hp}(HPL(0), HPL(B))$} 
and \hbox{$K(L)=exp(d_{{\mathcal T}^L}(ic_1(L),0)$}.

\vspace{.5in}
\begin{framed}

\begin{verbatim}
L=1.48
G(L) = 1.327185362837166
HPL(0) = 0.000007509959438 + 0.009347547230674i
HPL(B) =  0.000009420062234 + 0.067016970686742i
H(L) = 1.969831901361628
K(L) = 7.169471208698489
\end{verbatim}
\end{framed}

\affiliationone{
   M. Bridgeman and A. Yarmola\\
   Department of Mathematics\\
Boston College\\
Chestnut Hill, Ma 02467\\
   USA
   \email{\\
   bridgem@bc.edu\\
   yarmola@bc.edu}}
\affiliationtwo{
   R. Canary\\
      Department of Mathematics\\
University of Michigan\\
   Ann Arbor, MI 41809\\
    USA
   \email{canary@umich.edu}}

\end{document}